\documentclass{article}%
\usepackage{amsmath}
\usepackage{amsfonts}
\usepackage{amssymb}
\usepackage{graphicx}%
\providecommand{\U}[1]{\protect\rule{.1in}{.1in}}
\newtheorem{theorem}{Theorem}

\newtheorem{lemma}[theorem]{Lemma}

\newtheorem{proposition}[theorem]{Proposition}

\newenvironment{proof}[1][Proof]{\noindent\textbf{#1.} }{\ \rule{0.5em}{0.5em}}
\begin{document}

\title{Defining $\mathbb{A}$ in $G(\mathbb{A})$}
\author{Dan Segal}
\maketitle

\bigskip It is shown in the papers \cite{NST} and \cite{ST} that for many
integral domains $R,$ the ring is bi-intepretable with various Chevalley
groups $G(R)$. The model theory of ad\`{e}le rings and some of their subrings
has attracted some recent interest (\cite{DM}, \cite{D}, \cite{AMO}), and it seemed
worthwhile to extend the results in that direction.

Let $\mathbb{A}$ denote the ad\`{e}le ring of a global field $K$, with
$\mathrm{char}(K)\neq2,3,5$. We consider subrings of $\mathbb{A}$ of the
following kind:%
\begin{align*}
R  &  =\mathbb{A},\\
R  &  =~\prod_{\mathfrak{p}\in\mathcal{P}}\mathfrak{o}_{\mathfrak{p}}%
\end{align*}
where $\mathfrak{o}$ is the ring of integers of $K$ and $\mathcal{P}$ may be
any non-empty set of primes (or places) of $K$. For example, $R$ could be the
whole ad\`{e}le ring of $\mathbb{Q}$, or $\widehat{\mathbb{Z}}=\prod
_{p}\mathbb{Z}_{p}.$

\begin{theorem}
\label{sl2}The ring $R$ is bi-interpretable with each of the groups
$\mathrm{SL}_{2}(R)$, $\mathrm{SL}_{2}(R)/\left\langle -1\right\rangle ,$
$\mathrm{PSL}_{2}(R)$.
\end{theorem}

\begin{theorem}
\label{high_rank}Let $G$ be a simple Chevalley-Demazure group scheme of rank
at least $2$. Then $R$ is bi-interpretable with the group $G(R)$.
\end{theorem}

The special cases where $R=\mathfrak{o}_{\mathfrak{p}}$ were established in
\cite{NST}, \S 4 and \cite{ST}.

\bigskip

For a rational prime $p$ we write $R_{p}=\prod_{\mathfrak{p}\in\mathcal{P}%
,~\mathfrak{p}\mid p}\mathfrak{o}_{\mathfrak{p}}$.

\begin{lemma}
\label{S-lemma}$R$ has a finite subset $S$ such that every element of $R$ is
equal to one of the form%
\begin{equation}
\xi^{2}-\eta^{2}+s \label{formula}%
\end{equation}
with $\xi,\eta\in R^{\ast}$ and $s\in S$.
\end{lemma}

\begin{proof}
In any field of characteristic not $2$ and size $>5$, every element is the
difference of two non-zero squares. It follows that the same is true for each
of the rings $\mathfrak{o}_{\mathfrak{p}}$ with $N(\mathfrak{p})>5$ and odd.

If $N(\mathfrak{p})$ is $3$ or $5$ then every element of $\mathfrak{o}%
_{\mathfrak{p}}$ is of the form (\ref{formula}) with $\xi,\eta\in
\mathfrak{o}_{\mathfrak{p}}^{\ast}$ and $s\in\{0,\pm1\}.$ If $\mathfrak{p}$
divides $2$, the same holds if $S$ is a set of representatives for the cosets
of $4\mathfrak{p}$ in $\mathfrak{o}$.

Now by the Chinese Remainder Theorem (and Hensel's lemma) we can pick a finite
subset $S_{1}$ of $R_{2}\times R_{3}\times R_{5}$ such that every element of
$R_{2}\times R_{3}\times R_{5}$ is of the form (\ref{formula}) with $\xi
,\eta\in\mathfrak{o}_{\mathfrak{p}}^{\ast}$ and $s\in S_{1}$. Finally, let $S$
be the subset of elements $s\in R$ that project into $S_{1}$ and have
$\mathfrak{o}_{\mathfrak{p}}$-component $1$ for all $\mathfrak{p}\nmid30$
(including infinite places if present).
\end{proof}

\medskip

\textbf{Remark} If $K=\mathbb{Q}$ one could choose $S\subset\mathbb{Z}$
(diagonally embedded in $R$). The plethora of parameters in the following
argument can then be replaced by just three - $h(\tau),~u(1),$ $v(1)$ - or
even two when $R=\mathbb{A}$, in which case we replace $h(\tau)$ by $h(2),$
which can be expressed in terms of $u(1)$ and $v(1)$ by the formula
(\ref{h-identity}) below. Also the formula (\ref{P-formula}) can be replaced
by the simpler one: $y_{2}=u^{x}u^{-y}u^{s}\wedge y_{3}=y_{1}^{x}y_{1}%
^{-y}y_{1}^{s}$.

\medskip

For a finite subset $T$ of $\mathbb{Z}$ let
\[
R_{T}=\left\{  r\in R~\mid~r_{\mathfrak{p}}\in T\text{ for every }%
\mathfrak{p}\right\}  .
\]
This is a definable set, since $r\in R_{T}$ if and only if $f(r)=0$ where
$f(X)=\prod_{t\in T}(X-t)$.

\medskip

Choose $S$ as in Lemma \ref{S-lemma}, with $0,~1\in S$, and write $S^{2}=S.S.$

\medskip

Let $\Gamma=\mathrm{SL}_{2}(R)/Z$ where $Z$ is $1$, $\left\langle
-1\right\rangle $ or the centre of $\mathrm{SL}_{2}(R)$. \ For $\lambda\in R$
write%
\[
u(\lambda)=\left(
\begin{array}
[c]{cc}%
1 & \lambda\\
0 & 1
\end{array}
\right)  ,~v(\lambda)=\left(
\begin{array}
[c]{cc}%
1 & 0\\
-\lambda & 1
\end{array}
\right)  ,~~h(\lambda)=\left(
\begin{array}
[c]{cc}%
\lambda^{-1} & 0\\
0 & \lambda
\end{array}
\right)  ~~(\lambda\in R^{\ast})
\]
(matrices interpreted modulo $Z$; note that $\lambda\longmapsto u(\lambda)$ is
bijective for each choice of $Z$).

Fix $\tau\in R^{\ast}$ with $\tau_{\mathfrak{p}}=2$ for $\mathfrak{p}\nmid2,$
$\tau_{\mathfrak{p}}=3$ for $\mathfrak{p}\mid2$. It is easy to verify that%
\begin{equation}
\mathrm{C}_{\Gamma}(h(\tau))=h(R^{\ast}):=H. \label{Hdef}%
\end{equation}

\begin{proposition}
\label{P1}The ring $R$ is definable in $\Gamma.$
\end{proposition}

\begin{proof}
We take $h:=h(\tau)$ and $\{u(c)~\mid~c\in S^{2}\}$ as parameters, and put
$u:=u(1)$. `Definable' will mean definable with these parameters. For
$\lambda\in R$ and $\mu\in R^{\ast}$ we have%
\[
u(\lambda)^{h(\mu)}=u(\lambda\mu^{2}).
\]

Now (\ref{Hdef}) shows that $H$ is definable. If $\lambda=\xi^{2}-\eta^{2}+s$
and $x=h(\xi),~y=h(\eta)$ then%
\[
u(\lambda)=u^{x}u^{-y}u(s).
\]
It follows that%
\[
U:=u(R)=\bigcup_{s\in S}\{u^{x}u^{-y}u(s)~\mid~x,~y\in H\}
\]
is definable.

The map $u:R\rightarrow U$ is an isomorphism from $(R,+)$ to $U$. It becomes a
ring isomorphism with multiplication $\ast$ if one defines%
\begin{equation}
u(\beta)\ast u(\alpha)=u(\beta\alpha). \label{thing}%
\end{equation}
We need to provide an $L_{\mathrm{gp}}$ formula $P$ such that for
$y_{1},~y_{2},~y_{3}\in U$,%
\begin{equation}
y_{1}\ast y_{2}=y_{3}\Longleftrightarrow\Gamma\models P(y_{1},y_{2},y_{3}).
\label{nextthing}%
\end{equation}

Say $\alpha=\xi^{2}-\eta^{2}+s$, $\beta=\zeta^{2}-\rho^{2}+t$. Then%
\[
u(\beta\alpha)=u(\beta)^{x}u(\beta)^{-y}u(s)^{z}u(s)^{-r}u(st)
\]
where $x=h(\xi),~y=h(\eta),~z=h(\zeta)~$and $r=h(\rho)$.

So we can take $P(y_{1},y_{2},y_{3})$ to be a formula expressing the
statement: there exist $x,~y,z,r\in H$ such that for some $s,t\in S$%
\begin{align}
y_{1}  &  =u^{z}u^{-r}u(t),~y_{2}=u^{x}u^{-y}u(s),\label{P-formula}\\
y_{3}  &  =y_{1}^{x}y_{1}^{-y}u(s)^{z}u(s)^{-r}u(st).\nonumber
\end{align}

\end{proof}

\begin{proposition}
The group $\Gamma$ is interpretable in $R$.
\end{proposition}

\begin{proof}
When $\Gamma=\mathrm{SL}_{2}(R)$, \ clearly $\Gamma$ is definable as the set
of $2\times2$ matrices with determinant $1$ and group operation matrix
multiplication. For the other cases, it suffices to note that the equivalence
relation `modulo $Z$' is definable by $A\thicksim B$ iff there exists
$Z\in\{\pm1_{2}\}$ with $B=AZ$, resp. $Z\in H$ with $Z^{2}=1$ and $B=AZ$.
\end{proof}

\bigskip

To complete the proof of Theorem \ref{sl2} it remains to establish
\textbf{Step 1 }and \textbf{Step 2 }below.

We take $v=v(1)$ as another parameter, and set $w=uvu=\left(
\begin{array}
[c]{cc}%
0 & 1\\
-1 & 0
\end{array}
\right)  .$ Then $u(\lambda)^{w}=v(\lambda)$, so $V:=v(R)=U^{w}$ is definable.
Note the identity (for $\xi\in R^{\ast}$):%
\begin{equation}
h(\xi)=v(\xi)u(\xi^{-1})v(\xi)w^{-1}=w^{-1}u(\xi)w.u(\xi^{-1}).w^{-1}%
u(\xi).\label{h-identity}%
\end{equation}

\noindent\textbf{Step 1:} The ring isomorphism from $R$ to $U\subset
\mathrm{M}_{2}(R)$ is definable. Indeed, this is just the mapping%
\[
r\longmapsto\left(
\begin{array}
[c]{cc}%
1 & r\\
0 & 1
\end{array}
\right)  .
\]

\noindent\textbf{Step 2:} The map $\theta$ sending $g=(a,b;c,d)$ to
$(u(a),u(b);u(c),u(d))\in\Gamma^{4}$ is definable; this is a group isomorphism
when $U$ is identified with $R$ via $u(\lambda)\longmapsto\lambda$.

Assume for simplicity that $\Gamma=\mathrm{SL}_{2}(R)$. We start by showing
that the restriction of $\theta$ to each of the subgroups $U,~V\,,~H$ is
definable. Recall that $u(0)=1$ and $u(1)=u$.

If $g\in U$ then $g\theta=(u,g;1,u)$. If $g=v(-\lambda)\in V$ then
$g^{-w}=u(\lambda)\in U$ and $g\theta=(u,1;g^{-w},u).$

Suppose $g=h(\xi)\in H$. Then $g=w^{-1}xwyw^{-1}x$ where $x=u(\xi),$
$y=u(\xi^{-1}),$ and $g\theta=(y,1;1,x)$. So $g\theta=(y_{1},y_{2};y_{3}%
,y_{4})$ if and only if
\begin{align*}
~y_{4}\ast y_{1}  &  =u,~y_{2}=y_{3}=1,\\
g  &  =w^{-1}y_{4}wy_{1}w^{-1}y_{4}.
\end{align*}
Thus the restriction of $\theta$ to $H$ is definable.

Next, set%
\[
W:=\left\{  x~\in\Gamma\mid x_{\mathfrak{p}}\in\{1,w\}\text{ for every
}\mathfrak{p}~\right\}  .
\]
To see that $W$ is definable, observe that an element $x$ is in $W$ if and
only if there exist $y,z\in u(R_{\{0,1\}})$ such that
\[
x=yz^{w}y\text{ and }x^{4}=1.
\]
Note that $u(R_{\{0,1\}})$ is definable by (the proof of) Proposition \ref{P1}.

Put%
\[
\Gamma_{1}=\{g\in\Gamma\mid g_{11}\in R^{\ast}\}.
\]

If $g=(a,b;c,d)\in\Gamma_{1}$ then $g=\widetilde{v}(g)\widetilde{h}%
(g)\widetilde{u}(g)$ where%
\begin{align*}
\widetilde{v}(g)  &  =v(-a^{-1}c)\in V\\
\widetilde{h}(g)  &  =h(a^{-1})\in H\\
\widetilde{u}(g)  &  =u(a^{-1}b)\in U.
\end{align*}
This calculation shows that in fact $\Gamma_{1}=VHU$, so $\Gamma_{1}$ is
definable; these three functions on $\Gamma_{1}$ are definable since%
\begin{align*}
x  &  =\widetilde{v}(g)\Longleftrightarrow x\in V\cap HUg\\
y  &  =\widetilde{u}(g)\Longleftrightarrow y\in U\cap HVg\\
z  &  =\widetilde{h}(g)\Longleftrightarrow z\in H\cap VgU.
\end{align*}

Let $g=(a,b;c,d)$. Then $gw=(-b,a;-d,c)$. We claim that there exists $x\in W$
such that $gx\in\Gamma_{1}$. Indeed, this may be constructed as follows: If
$a_{\mathfrak{p}}\in\mathfrak{o}_{\mathfrak{p}}^{\ast}$ take $x_{\mathfrak{p}%
}=1$. If $a_{\mathfrak{p}}\in\mathfrak{po}_{\mathfrak{p}}$ and
$b_{\mathfrak{p}}\in\mathfrak{o}_{\mathfrak{p}}^{\ast}$ take $x_{\mathfrak{p}%
}=w$. If both fail, take $x_{\mathfrak{p}}=1$ when $a_{\mathfrak{p}}\neq0$ and
$x_{\mathfrak{p}}=w$ when $a_{\mathfrak{p}}=0$ and $b_{\mathfrak{p}}\neq0$.
This covers all possibilities since for almost all $\mathfrak{p}$ at least one
of $a_{\mathfrak{p}}$, $b_{\mathfrak{p}}$ is a unit in $\mathfrak{o}%
_{\mathfrak{p}},$ and $a_{\mathfrak{p}}$, $b_{\mathfrak{p}}$ are never both zero.

As $gx\in\Gamma_{1},$ we may write%
\[
gx=\widetilde{v}(gx)\widetilde{h}(gx)\widetilde{u}(gx)\text{.}%
\]

We claim that the restriction of $\theta$ to $W$ is definable. Let $x\in W$
and put $P=\{\mathfrak{p}~\mid~x_{\mathfrak{p}}=1\},$ $Q=\{\mathfrak{p}%
~\mid~x_{\mathfrak{p}}=w\}$. Then $(u^{x})_{\mathfrak{p}}$ is $u$ for
$\mathfrak{p}\in P$ and $v$ for $\mathfrak{p}\in Q$, so $u^{x}\in\Gamma_{1}$
and
\[
\widetilde{u}(u^{x})_{\mathfrak{p}}=\left\{
\begin{array}
[c]{ccc}%
u &  & (\mathfrak{p}\in P)\\
1 &  & (\mathfrak{p}\in Q)
\end{array}
\right.  .
\]
Recalling that $u=u(1)$ and $1=u(0)$ we see that%
\[
x\theta=\left(
\begin{array}
[c]{cc}%
\widetilde{u}(u^{x}) & \widetilde{u}(u^{x})^{-1}u\\
u^{-1}\widetilde{u}(u^{x}) & \widetilde{u}(u^{x})
\end{array}
\right)  .
\]

We can now deduce that $\theta$ is definable. Indeed, $g\theta=A$ holds if and
only if there exists $x\in W$ such that $gx\in\Gamma_{1}$ and
\[
A.x\theta=\widetilde{v}(gx)\theta.\widetilde{h}(gx)\theta.\widetilde{u}%
(gx)\theta
\]
(of course the products here are matrix products, definable in the language of
$\Gamma$ in view of Proposition \ref{P1}).

This completes the proof of Theorem \ref{sl2}\ for $\Gamma=\mathrm{SL}_{2}%
(R)$. When $\Gamma=\mathrm{SL}_{2}(R)/Z$, the same formulae now define
$\theta$ as a map from $\Gamma$ into the set of $2\times2$ matrices with
entries in $U$ modulo the appropriate definable equivalence relation.
$\blacksquare$

\bigskip Now we turn to the proof of Theorem \ref{high_rank}. This largely
follows \cite{ST}, \S \S 3, 4, but is simpler because we are dealing here with
`nice' rings. Henceforth $G$ denotes a simple Chevalley-Demazure group scheme
of rank at least $2$. The root subgroup associated to a root $\alpha$ is
denoted $U_{\alpha}$, and $Z$ denotes the centre of $G$. Put $\Gamma=G(R)$.

Let $S$ be any integral domain with infinitely many units. According to
\cite{ST}, Theorem 1.5 we have%
\[
U_{\alpha}(S)Z(S)=\mathrm{Z}\left(  C_{G(S)}(v)\right)
\]
whenever $1\neq v\in U_{\alpha}(S).$ This holds in particular for the rings
$S=\mathfrak{o}_{\mathfrak{p}}$. Take $u_{\alpha}\in U_{\alpha}(R)$ to have
$\mathfrak{p}$-component $x_{\alpha}(1)$ for each $\mathfrak{p}\in\mathcal{P}$
(or every $\mathfrak{p}$ when $R=\mathbb{A}$); then%
\[
U_{\alpha}(R)Z(R)=\mathrm{Z}\left(  C_{G(R)}(u_{\alpha})\right)  .
\]
Given this, the proof of Corollary 1.6 of \cite{ST} now shows that $U_{\alpha
}(R)$ is a definable subgroup of $\Gamma$; the result is stated for integral
domains but the argument remains valid, noting that in the present case $R/2R$
is finite.

Associated to each root $\alpha$ there is a morphism $\varphi_{\alpha
}:\mathrm{SL}_{2}\rightarrow G$ sending $u(r)=\left(
\begin{array}
[c]{cc}%
1 & r\\
0 & 1
\end{array}
\right)  $ to $x_{\alpha}(r)$ and $v(r)=\left(
\begin{array}
[c]{cc}%
1 & 0\\
r & 1
\end{array}
\right)  $ to $x_{-\alpha}(r)$ (\cite{S}, Chapter 3). This morphism is defined
over $\mathbb{Z}$ and satisfies
\[
K_{\alpha}:=\mathrm{SL}_{2}(R)\varphi_{\alpha}\leq G(R).
\]

\begin{lemma}
\label{Klemma}$K_{\alpha}=U_{-\alpha}(R)U_{\alpha}(R)U_{-\alpha}(R)U_{\alpha
}(R)U_{-\alpha}(R)U_{\alpha}(R)U_{-\alpha}(R)U_{\alpha}(R).$
\end{lemma}

\begin{proof}
This follows from the corresponding identity in $\mathrm{SL}_{2}(R),$ which in
turn follows from (\ref{h-identity}) and the fact that $w=uvu$.
\end{proof}

\bigskip

We may thus infer that each $K_{\alpha}$ is a definable subgroup of $G(R)$.
Fixing a root $\gamma$, we identify $R$ with $U_{\gamma}(R)$ by $r\longmapsto
r^{\prime}=x_{\gamma}(r).$ Proposition \ref{P1} now shows that $R$ is
definable in $G(R)$.

As above, $G(R)$ is $R$-definable as a set of $d\times d$ matrices that
satisfy a family of polynomial equations over $\mathbb{Z}$, with group
operation matrix multiplication.

To complete the proof we need to establish

\medskip\textbf{1.} The ring isomorphism $R\rightarrow U_{\gamma
}(R);~r\longmapsto r^{\prime}=x_{\gamma}(r)\in\mathrm{M}_{d}(R)$ is definable
in ring language. This follows from the definition%
\[
x_{\gamma}(r)=\exp(rX_{\gamma})=1+rM_{1}(\gamma)+\ldots+r^{q}M_{q}(\gamma)
\]
where each $M_{i}(\gamma)$ is a matrix with integer entries (\cite{S}, Chaps.
2, 3).

\bigskip

\textbf{2.} The group isomorphism $\theta:G(R)\rightarrow G(R^{\prime})
\subseteq\mathrm{M}_{d}(U_{\gamma}(R))$ is definable in group language.

\medskip

To begin with, Lemma 3.5 of \cite{ST} shows that for each root $\alpha$, the
restriction of $\theta$ to $U_{a}(R)$ is definable (this is established for
$R$ an integral domain, but the proof is valid in general). Next, we observe
that $G(R)$ has `finite elementary width' in the sense of \cite{ST}:

\begin{lemma}
There is is finite sequence of roots $\beta_{i}$ such that%
\[
G(R)=\prod_{i=1}^{N}U_{\beta_{i}}(R).
\]

\end{lemma}

\begin{proof}
This relies on results from Chapter 7 of \cite{S}. Specifically, Corollary 2
to Theorem 18 asserts that if $R$ is a PID, then (in the above notation)
$G(R)$ is generated by the groups $K_{\alpha}$. It is clear from the proof
that each element of $G(R)$ is in fact a product of bounded length of elements
from various of the $K_{\alpha}$; an upper bound is given by the sum $N_{1}$,
say, of the following numbers: the number of positive roots, the number of
fundamental roots, and the maximal length of a Weyl group element as a product
of fundamental reflections. If the positive roots are $\alpha_{1}%
,\ldots,\alpha_{n}$ it follows (if $R$ is a PID) that
\[
G(R)=\left(  \prod_{j=1}^{n}K_{\alpha_{j}}\right)  \cdot\left(  \prod
_{j=1}^{n}K_{\alpha_{j}}\right)  \cdot\ldots\cdot\left(  \prod_{j=1}%
^{n}K_{\alpha_{j}}\right)  ~\ \text{(}N_{1}\text{ factors).}%
\]
As each of the rings $\mathfrak{o}_{\mathfrak{p}}$ is a PID (or a field), the
same holds for our ring $R$ in general.

The result now follows by Lemma \ref{Klemma}, taking $N=8nN_{1}$.
\end{proof}

\bigskip

Thus $\theta$ is definable as follows: for $g\in G(R)$ and $A\in\mathrm{M}%
_{d}(U_{\gamma}(R)),$ $g\theta=A$ if and only if there exist $v_{i}\in
U_{\beta_{i}}(R)$ and $A_{i}\in\mathrm{M}_{d}(U_{\gamma}(R))$ such that
$g=v_{1}\ldots v_{N}$, $A=A_{1}\cdot\ldots\cdot A_{N}$ and $A_{i}=v_{i}\theta$
for each $i$. Here $A_{1}\cdot A_{2}$ etc denote matrix products, which are
definable in the language of $G$ because the ring operations on $R^{\prime
}=U_{\gamma}(R)$ are definable in $G$.

This completes the proof.

\medskip

\emph{Acknowledgment.} \ Thanks to Jamshid Derakhshan for references and advice.

\end{document}